\documentclass[]{article}
\usepackage[utf8]{inputenc}
\usepackage[english]{babel}
\usepackage{amsthm}
\usepackage{amsfonts}
\usepackage{hyperref}
\usepackage{amsmath}
\newtheorem{theorem}{Theorem}
\newtheorem{corollary}{Corollary}
\newtheorem{lemma}{Lemma}

\title{Regular grid subgraphs of maximal girth}
\author{
\textsc{Jan Kristian Haugland}\\
\texttt{admin@neutreeko.net}}

\begin{document}

\maketitle

\begin{abstract}
\noindent The unit-distance graph on the $n$-dimensional integer lattice $\mathbb{Z}^n$ is called the $n$-dimensional grid. We attempt to maximize the girth of a $k$-regular (possibly induced) subgraph of the $n$-dimensional grid, and provide examples and bounds for selected values of $n$ and $k$, along with more general results. A few cases involving alternative lattices are also considered.
\end{abstract}

\section{Introduction}
If $G$ is a graph, the \textit{girth} of $G$ is the length of its shortest cycle, and is denoted by $g(G)$. 

An \textit{induced} subgraph $H$ of a graph $G$ contains an edge between any two vertices in its vertex set if they are adjacent in $G$.

If $\Lambda$ is a point lattice, the minimum distance graph of $\Lambda$ is called the \textit{grid} associated with $\Lambda$. For a given point lattice and some integer $k \geq 3$, we can ask for the maximal girth of a $k$-regular (possibly induced) subgraph of the associated grid. In \cite{Haugland:2003}, the author solved the case with induced subgraphs for the point lattice $\mathbb{Z}^3$ and $k = 3$, giving the four different isometry classes of subgraphs of girth 10. The vertex sets of four subgraphs representing those classes are given by $$v(G_1) = \left\{(x, y, z) \in \mathbb{Z}^3 \textrm{ }|\textrm{ } 2x \equiv z \textrm{ (mod 4)} \textrm{ or } 2y \equiv z - 1 \textrm{ (mod 4)}\right\}$$
$$v(G_2) = \left\{ \begin{split}
(x, y, z) \in \mathbb{Z}^3 \textrm{ }|\textrm{ }x + 3y + 5z \geq 3 \textrm{ and } \equiv 2, 3, 4 \textrm{ or } 6 \textrm{ (mod 7)} \\
\textrm{ or } x + 3y + 5z \leq 0 \textrm{ and } \equiv 0, 1, 4 \textrm{ or } 6 \textrm{ (mod 7)}
\end{split} \right\}$$
$$v(G_3) = \left\{(x, y, z) \in \mathbb{Z}^3 \textrm{ }|\textrm{ }x + 2y + 3z \equiv 0, 1, 2 \textrm{ or } 4 \textrm{ (mod 7)}\right\}$$
$$v(G_4) = \left\{ \begin{split}
(x, y, z) \in \mathbb{Z}^3 \textrm{ }|\textrm{ }(x, y, z) \textrm{ or } (x + 2, y + 2, z + 2) \equiv (0, 0, 0), \\
(2, 0, 0), (3, 0, 0), (0, 1, 0), (2, 1, 0), (0, 2, 0), (1, 2, 0), (2, 2, 0), \\
(1, 3, 0), (3, 3, 0), (0, 0, 1), (1, 0, 1), (1, 1, 1), (2, 1, 1), (3, 1, 1), \\
(0, 2, 1), (3, 2, 1), (1, 3, 1), (2, 3, 1) \textrm{ or } (3, 3, 1) \textrm{ (mod 4)}
\end{split} \right\}$$

This paper presents various findings in other cases. For the most part, $\Lambda = \mathbb{Z}^n$ for some integer $n \geq 2$, in which case we simply refer to the associated grid as the $n$-dimensional grid.

The point lattices BCC (body-centred cubic), FCC (face-centred cubic) and $D_4$ will also be considered (at least briefly). The BCC lattice can be defined as $$\{(x, y, z) \in \mathbb{Z}^3 \textrm{ }|\textrm{ } x \equiv y \equiv z \textrm{ (mod 2)}\}$$ with minimum distance $\sqrt 3$, and the FCC lattice can be defined as $$\{(x, y, z) \in \mathbb{Z}^3 \textrm{ }|\textrm{ } x + y + z \equiv 0 \textrm{ (mod 2)}\}$$ with minimum distance $\sqrt 2$. An example from the literature of an induced subgraph of the grid associated with the FCC lattice is given by the Laves graph, as described in \cite{Coxeter:1955}. (Incidentally, the Laves graph is isomorphic to the graph $G_3$ given above.) $D_4$ can be defined as $$\{(x_1, x_2, x_3, x_4) \in \mathbb{Z}^4 \textrm{ }|\textrm{ } x_1 + x_2 + x_3 + x_4 \equiv 0 \textrm{ (mod 2)}\}$$ with minimum distance $\sqrt{2}$. The map $$(x_1, x_2, x_3, x_4) \rightarrow (x_1 - x_2, x_1 + x_2, x_3 - x_4, x_3 + x_4)$$ yields the point lattice $$\{(x_1, x_2, x_3, x_4) \in \mathbb{Z}^4 \textrm{ }|\textrm{ } x_1 \equiv x_2 \equiv x_3 \equiv x_4 \textrm{ (mod 2)}\}$$ and thus, $D_4$ can be thought of as a 4-dimensional analogue of both the BCC lattice and the FCC lattice.

It is assumed throughout that the null graph is not a valid example of a k-regular graph.

A nonconventional concept that we introduce here, related to the girth, is that of the \textit{spread} of a graph $G$ at depth $d$. It is defined as the number of vertices at a graph distance $d$ from a given vertex $v_0$, minimized over all vertices $v_0$, and is denoted by $s(G, d)$. If $G$ is $k$-regular, then $g(G) \geq 2d+1$ if and only if $s(G, d) = k (k-1)^{d-1}$. Thus, the value of $s(G, d)$ can be used as a "tiebreaker" between $k$-regular subgraphs of the same girth $g \in \{2d - 1, 2d\}$.

\section{Induced subgraphs}

\begin{theorem}
	Suppose $n$ and $k$ are integers satisfying $k > n \geq 2$. If $n$ is odd and $k = n+1$, then the maximal girth of a $k$-regular induced subgraph of the $n$-dimensional grid is 6, which is attained by only one such subgraph up to isometry. Otherwise, the maximal girth is 4.
\end{theorem}

\begin{proof}
	If there are two adjacent vertices $u$, $v$ that in total have more than $2n-2$ neighbours not on the line through $u$ and $v$, then it is immediate by the pigeonhole principle that there is a cycle of length 4 through $u$ and $v$. This is the case for any pair of neighbours if $k$ is at least $n+2$, and then the girth must be equal to 4.
	
	If $k = n+1$, then it may be possible to avoid such pairs, but only if the neighbours of each vertex match up in antipodal pairs, which means that $k$ must be even, and so $n$ must be odd. Otherwise, the girth must still be equal to 4.
	
	If $n$ is odd and $k = n+1$, let us assume that the girth is $\geq 6$. Given two adjacent vertices $u$ and $v$, the placement of $v$'s remaining neighbours is uniquely determined by the placement of $u$'s remaining neighbours. It follows that there is at most one $(n+1)$-regular induced subgraph of the $n$-dimensional grid of girth $\geq 6$ up to isometry.
	
	And such a subgraph exists: Let $(x_1, ..., x_n) \in \mathbb{Z}^n$ be included in the vertex set if and only if $\sum (-1)^{x_i} = \pm 1$. It can be verified that the girth is equal to 6.
\end{proof}

\begin{lemma}
	If $n$, $m$ and $k$ are integers with $n \geq m \geq 1$ and $0 \leq k \leq 2n$, and $\Gamma$ is a $k$-regular subgraph of the $n$-dimensional grid, then there exists a unit m-hypercube in the grid that contains at least one vertex from $\Gamma$, and for which the ratio between the number of edges and the number of vertices in its intersection with $\Gamma$ is at least ${km \over 4n}$. Moreover, if it is known that the ratio $\leq {km \over 4n}$ for all unit $m$-hypercubes containing at least one vertex of $\Gamma$, then equality must hold in general, except possibly for a subset of the unit $m$-hypercubes of density 0.
\end{lemma}

\begin{proof}
	Let $S$ denote the set of unit $m$-hypercubes for which the intersection with $\Gamma$ contains at least one vertex. Let $v$ and $e$ denote the expected number of vertices and the expected number of edges, respectively, in the intersection between $\Gamma$ and a unit $m$-hypercube in $S$. Each vertex in $\Gamma$ belongs to ${n\choose m} 2^m$ unit $m$-hypercubes in $S$, and each edge in $\Gamma$ belongs to ${n - 1 \choose m - 1} 2^{m - 1}$ unit $m$-hypercubes in $S$. Each vertex is incident with $k$ edges while each edge is incident with 2 vertices, and hence we have $${e \over v} = {k \over 2} {{n - 1 \choose m - 1} 2^{m - 1} \over {n\choose m} 2^m} = {km \over 4n}$$ and the conclusion follows.
\end{proof}

\begin{lemma}
	Suppose $\Gamma$ is a 6-regular induced subgraph of the 6-dimensional grid containing a vertex $u$ such that the neighbours of $u$ do not match up in antipodal pairs. Then the girth of $\Gamma$ is at most 6.
\end{lemma}

\begin{proof}
	We can assume without loss of generality that $u = (0, 0, 0, 0, 0, 0)$, and that $v = (1, 0, 0, 0, 0, 0)$ is in the vertex set of $\Gamma$ while $(-1, 0, 0, 0, 0, 0)$ is not. Thus, $u$ has five other neighbours $u_i = (0, a_{i_2}, a_{i_3}, a_{i_4}, a_{i_5}, a_{i_6})$ with one $a_{i_k} = \pm 1$ and the others equal to 0, for each $i \in \{1, ..., 5\}$. The point (2, 0, 0, 0, 0, 0) may or may not be in the vertex set of $\Gamma$, and $v$ has four or five other neighbours $v_j = (1, b_{j_2}, b_{j_3}, b_{j_4}, b_{j_5}, b_{j_6})$ accordingly.
	
	If $a_{i_k} = b_{j_k}$ for all $k \in \{2, ..., 6\}$ for some $(i, j)$, then the girth is 4, so we will assume that this is not the case henceforth. Likewise, we can assume that there are no points on the form $$(0, a_{i_2^{(1)}} + a_{i_2^{(2)}}, a_{i_3^{(1)}} + a_{i_3^{(2)}}, a_{i_4^{(1)}} + a_{i_4^{(2)}}, a_{i_5^{(1)}} + a_{i_5^{(2)}}, a_{i_6^{(1)}} + a_{i_6^{(2)}})$$ (provided $u_{i^{(1)}} \neq \pm u_{i^{(2)}}$) in the vertex set of $\Gamma$, nor any on the form $$(1, b_{j_2^{(1)}} + b_{j_2^{(2)}}, b_{j_3^{(1)}} + b_{j_3^{(2)}}, b_{j_4^{(1)}} + b_{j_4^{(2)}}, b_{j_5^{(1)}} + b_{j_5^{(2)}}, b_{j_6^{(1)}} + b_{j_6^{(2)}})$$ (with corresponding exceptions).
	
	If $(2, 0, 0, 0, 0, 0)$ is not in the vertex set of $\Gamma$, then each vertex $u_i$ has at least three neighbours on the form $$(0, a_{i_2} + b_{j_2}, a_{i_3} + b_{j_3}, a_{i_4} + b_{j_4}, a_{i_5} + b_{j_5}, a_{i_6} + b_{j_6})$$ other than $u$. This is because the only possible neighbours $u_i$ could have that are not on this form are $(-1, a_{i_2}, a_{i_3}, a_{i_4}, a_{i_5}, a_{i_6})$ and $(0, 2a_{i_2}, 2a_{i_3}, 2a_{i_4}, 2a_{i_5}, 2a_{i_6})$, in addition to $u$. Thus, there are at least 15 vertices on this form. Likewise, each vertex $v_j$ has at least three neighbours on the form $$(1, a_{i_2} + b_{j_2}, a_{i_3} + b_{j_3}, a_{i_4} + b_{j_4}, a_{i_5} + b_{j_5}, a_{i_6} + b_{j_6})$$ other than $v$, for a total of at least 15 vertices. With at most 24 possible combinations $$(a_{i_2} + b_{j_2}, a_{i_3} + b_{j_3}, a_{i_4} + b_{j_4}, a_{i_5} + b_{j_5}, a_{i_6} + b_{j_6})$$ ($5^2$ minus at least one combination that yields $(0, 0, 0, 0, 0)$), it follows that there must exist a cycle of length at most 6.
	
	The remaining case is when (2, 0, 0, 0, 0, 0) is in the vertex set of $\Gamma$. Now there are only four vertices $v_j$, but we can assume that neither one of them has a neighbour $(2, b_{j_2}, b_{j_3}, b_{j_4}, b_{j_5}, b_{j_6})$. Each vertex $u_i$ now has at least two neighbours on the form $$(0, a_{i_2} + b_{j_2}, a_{i_3} + b_{j_3}, a_{i_4} + b_{j_4}, a_{i_5} + b_{j_5}, a_{i_6} + b_{j_6})$$ other than $u$, for a total of at least 10 vertices, and each vertex $v_j$ has at least three neighbours on the form $$(1, a_{i_2} + b_{j_2}, a_{i_3} + b_{j_3}, a_{i_4} + b_{j_4}, a_{i_5} + b_{j_5}, a_{i_6} + b_{j_6})$$ other than $v$, for a total of at least 12 vertices. There are at most 20 possible combinations $$(a_{i_2} + b_{j_2}, a_{i_3} + b_{j_3}, a_{i_4} + b_{j_4}, a_{i_5} + b_{j_5}, a_{i_6} + b_{j_6})$$ and again it follows that there must exist a cycle of length at most 6.
\end{proof}

\begin{theorem}
	Suppose $n$ is an integer $\geq 3$. The maximal girth of an $n$-regular induced subgraph of the $n$-dimensional grid is 10 for $n = 3$, 8 for $n = 4$ and $n = 6$, and 6 for all other integers $n \geq 3$. Moreover, the number of isometry classes of $n$-regular induced subgraphs of maximal girth is 4 for $n = 3$, 1 for $n = 6$, and unbounded for all other integers $n \geq 3$.
\end{theorem}

\begin{proof}
	The case $n = 3$ was solved in \cite{Haugland:2003}, as already mentioned.

	For $n = 4$, it follows from Lemma 1 that there must exist a unit 4-hypercube such that the intersection with the subgraph contains at least as many edges as vertices, which implies that it must contain a cycle. It is known that the maximum length of an induced cycle of a unit 4-hypercube is 8 (confer \cite{Klee:1970}). Here is a way to construct infinitely many different examples for which this is attained.

	For each $i \in \mathbb{Z}$, let $U_i$ be one of the sets
	$$\{ (x_1, x_2, x_3) \in \mathbb{Z}^3 \textrm{ }|\textrm{ } x_1 \equiv 0 \textrm{ (mod 2)}, x_2 \equiv 1 \textrm{ (mod 2)} \}$$
	$$\{ (x_1, x_2, x_3) \in \mathbb{Z}^3 \textrm{ }|\textrm{ } x_1 \equiv 1 \textrm{ (mod 2)}, x_2 \equiv 0 \textrm{ (mod 2)} \}$$
	$$\{ (x_1, x_2, x_3) \in \mathbb{Z}^3 \textrm{ }|\textrm{ } x_1 \equiv 0 \textrm{ (mod 2)}, x_3 \equiv 1 \textrm{ (mod 2)} \}$$
	$$\{ (x_1, x_2, x_3) \in \mathbb{Z}^3 \textrm{ }|\textrm{ } x_1 \equiv 1 \textrm{ (mod 2)}, x_3 \equiv 0 \textrm{ (mod 2)} \}$$
	$$\{ (x_1, x_2, x_3) \in \mathbb{Z}^3 \textrm{ }|\textrm{ } x_2 \equiv 0 \textrm{ (mod 2)}, x_3 \equiv 1 \textrm{ (mod 2)} \}$$
	$$\{ (x_1, x_2, x_3) \in \mathbb{Z}^3 \textrm{ }|\textrm{ } x_2 \equiv 1 \textrm{ (mod 2)}, x_3 \equiv 0 \textrm{ (mod 2)} \}$$
	\noindent such that $U_j \cap U_{j+1}$ is empty for any $j \in \mathbb{Z}$. Consider the induced subgraph with vertex set given by
	$$\{ (x_1, x_2, x_3, x_4) \in \mathbb{Z}^4 \textrm{ }|\textrm{ } x_1 \equiv x_2 \equiv x_3 \textrm{ (mod 2)} \textrm{ or } (x_1, x_2, x_3) \in U_{x_4} \}$$ The girth of this subgraph is 8, regardless of $\{U_i\}_{i \in \mathbb{Z}}$.

	Suppose $\Gamma^{(6)}$ is a 6-regular induced subgraph of the 6-dimensional grid of girth at least 8. By Lemma 2, the neighbours of each vertex match up in antipodal pairs. It follows that the intersection of $\Gamma^{(6)}$ with any unit 6-hypercube (if non-empty) must be 3-regular. As can be verified with computer assistance, there is only one 3-regular induced subgraph of the graph of a unit 6-hypercube of girth $\geq 8$ up to isometry. It is symmetrical around the centre, and it follows that the intersection with any unit 6-hypercube is unique given the intersection with any adjacent unit 6-hypercube. We conclude that there is only one such graph $\Gamma^{(6)}$ up to isometry. One possible representative has a vertex set consisting of all integer lattice points $(x_1, x_2, x_3, x_4, x_5, x_6) \in \mathbb{Z}^6$ that do \textit{not} satisfy any of the following sets of conditions:
	$$x_2 \not\equiv x_3 \equiv x_4 \equiv x_5 \not\equiv x_6 \textrm{ (mod 2)}$$
	$$x_3 \equiv x_4 \not\equiv x_5 \not\equiv x_6 \equiv x_1 \textrm{ (mod 2)}$$
	$$x_4 \not\equiv x_5 \equiv x_6 \equiv x_1 \not\equiv x_2 \textrm{ (mod 2)}$$
	$$x_5 \equiv x_6 \not\equiv x_1 \not\equiv x_2 \equiv x_3 \textrm{ (mod 2)}$$
	$$x_6 \not\equiv x_1 \equiv x_2 \equiv x_3 \not\equiv x_4 \textrm{ (mod 2)}$$
	$$x_1 \equiv x_2 \not\equiv x_3 \not\equiv x_4 \equiv x_5 \textrm{ (mod 2)}$$
	The girth of this graph is 8. The intersection with a unit 6-hypercube is isomorphic to the cubic symmetric graph F40A (the bipartite double cover of the dodecahedron graph of girth 8), and $\Gamma^{(6)}$ is a covering graph of F40A with double edges.

	The remaining cases are $n = 5$ and $n \geq 7$. The girth of the induced subgraph with vertex set
	$$V_k = \{ (x_1,..., x_n) \in \mathbb{Z}^n \textrm{ }|\textrm{ } x_1 + ... + x_n \in \{ k, k+1 \} \}$$
	is 6 for any $n \geq 3$ and any integer $k$, and we can easily find infinitely many isometrically distinct unions $V_{k_1} \cup V_{k_2} \cup ...$ that also induce $n$-regular subgraphs of girth 6, so it suffices to verify that it is not possible to attain a girth of 8 in those cases.

	Suppose $n \geq 5$, and that $\Gamma^{(5)}$ is an $n$-regular induced subgraph of the $n$-dimensional grid of girth $\geq 8$. The situation is quite similar to what we just considered for $n = k = 6$: There is only one non-empty induced subgraph $H$ of the unit 5-hypercube of girth $\geq 8$ with at least 5/4 as many edges as vertices, up to isometry. The intersection between $\Gamma^{(6)}$ and any unit 5-hypercube is isometric to $H$, and the degree of each vertex is either 2 or 3. By Lemma 1, all intersections between $\Gamma^{(5)}$ and a unit 5-hypercube must also be isometric to $H$, except possibly for a subset of the unit 5-hypercubes of density 0.
	
	When $n \geq 7$, then just by considering the unit 5-hypercubes that are incident with any given vertex in $\Gamma^{(5)}$, we can see that a positive portion of them contain at least one vertex of degree $\geq 4$ when intersected with $\Gamma^{(5)}$. Hence, they cannot all be isometric to $H$.
	
	When $n = 5$, it is not quite as simple to check that it is impossible that almost all intersections with a unit 5-hypercube is isometric to $H$, but it can be verified with computer assistance.

	This completes the proof, admittedly with some steps relying on computer verification.
\end{proof}

For $k < n$, most cases are open. We will take a look at some examples of $k$-regular induced subgraphs of high girth, even though maximality may not have been established as of yet.

For $n = 4$, $k = 3$, a computer search found a number of different examples of induced subgraphs of girth 12. One of them attains a spread of 94 at depth 6 (while a girth of at least 14 would give the value $3 \times 2^5 = 96$). The vertex set consists of
$$(x_1, x_2, x_3, x_4)$$
$$(x_1 + 1, x_2, x_3, x_4)$$
$$(x_1 + 1, x_2, x_3 + 1, x_4)$$
$$(x_1 + 2, x_2, x_3, x_4)$$
$$(x_1 + 2, x_2, x_3, x_4 + 1)$$
$$(x_1, x_2, x_3 - 1, x_4)$$
$$(x_1, x_2, x_3 - 2, x_4)$$
$$(x_1, x_2, x_3, x_4 - 1)$$
for all points $(x_1, x_2, x_3, x_4)$ in the point lattice spanned by \{[1, 1, 1, 1], \newline [1, 1, -1, -1], [1, 0, -1, 2], [3, 0, 1, 0]\}.

In view of how close this is, in a sense, to a girth of 14, it is perhaps not surprising that a girth of 14 can be attained if $\mathbb{Z}^4$ is replaced by $D_4$ as our point lattice $\Lambda$. In that case, one example of a cubic induced subgraph of girth 14 is the one with a vertex set consisting of
$$(x_1, x_2, x_3, x_4)$$
$$(x_1, x_2, x_3 + 1, x_4 - 1)$$
$$(x_1, x_2 - 1, x_3, x_4 + 1)$$
$$(x_1, x_2 + 1, x_3 - 1, x_4)$$
$$(x_1 - 1, x_2, x_3 + 1, x_4 - 2)$$
$$(x_1 + 1, x_2, x_3 + 2, x_4 - 1)$$
$$(x_1 - 1, x_2 - 1, x_3, x_4 + 2)$$
$$(x_1 + 1, x_2 - 3, x_3, x_4)$$
for all points $(x_1, x_2, x_3, x_4)$ in the point lattice spanned by \{[2, 2, 0, 0], \newline [0, 0, 2, 2], [2, -2, 3, -1], [2, -2, -1, 3]\}. The minimum distance between non-adjacent vertices is $\sqrt{6}$, while the lower bound for induced subgraphs in general is 2.

For $n = 5$, $k = 3$, there is an induced subgraph of girth 16. The vertex set is given by
$$(x_1, x_2, x_3, x_4, x_5)$$
$$(x_1 + 1, x_2, x_3, x_4, x_5)$$
$$(x_1 + 1, x_2, x_3 + 1, x_4, x_5)$$
$$(x_1 + 1, x_2, x_3 + 2, x_4, x_5)$$
$$(x_1, x_2 - 1, x_3, x_4, x_5)$$
$$(x_1, x_2 - 1, x_3, x_4 + 1, x_5)$$
$$(x_1, x_2 - 1, x_3, x_4 + 2, x_5)$$
$$(x_1, x_2 + 1, x_3, x_4, x_5)$$
$$(x_1, x_2 + 1, x_3, x_4, x_5 + 1)$$
$$(x_1, x_2 + 1, x_3, x_4, x_5 + 2)$$
$$(x_1, x_2 - 2, x_3, x_4 + 1, x_5)$$
$$(x_1, x_2 + 2, x_3, x_4, x_5 + 1)$$
for all points $(x_1, x_2, x_3, x_4)$ in the point lattice spanned by \{[2, 0, 1, 0, -1], \newline[2, 0, 0, -1, 1], [1, 3, -1, 0, 1], [0, 0, 1, 1, 1], [0, 0, 0, 3, 0]\}.

For $n = 5$, $k = 4$, it can be shown that the maximal girth is 10. The induced subgraph with vertex set given by
$$\{(x_1, x_2, x_3, x_4, x_5) \in \mathbb{Z}^5 \textrm{ }|\textrm{ } x_1 + 2x_2 + 3x_3 + 4x_4 + 5x_5 \equiv 1, 3, 4, 5 \textrm{ or } 9 \textrm{ (mod 11)}\}$$
is one example that attains a girth of 10. (The permissible residues are exactly the quadratic residues, but there are of course other, equivalent choices.) \newline Another example consists of all integer lattice points $(x_1, x_2, x_3, x_4, x_5) \in \mathbb{Z}^5$ for which
$$(x_1 - x_5, x_2 - x_5, x_3 - x_5, x_4 - x_5) \equiv (0, 0, 0, 0), (1, 0, 0, 0), (2, 0, 0, 0),$$ $$(1, 1, 0, 0), (1, 2, 0, 0), (1, 1, 1, 0), (0, 2, 1, 0), (0, 1, 2, 0), (1, 1, 2, 0), (2, 1, 2, 0),$$ $$(2, 0, 0, 1), (0, 2, 0, 1), (2, 0, 1, 1), (1, 1, 1, 1), (0, 2, 1, 1), (0, 0, 2, 1), (1, 0, 2, 1),$$ $$(2, 0, 2, 1), (0, 1, 2, 1), (0, 2, 2, 1), (2, 0, 0, 2), (2, 1, 0, 2), (2, 2, 0, 2), (1, 0, 1, 2),$$ $$(1, 1, 1, 2), (0, 2, 1, 2), (1, 2, 1, 2), (2, 2, 1, 2), (0, 1, 2, 2) \textrm{ or } (2, 2, 2, 2) \textrm{ (mod 3)}$$
It is a covering graph of the graph of the icosidodecahedron.

The claim that 10 is maximal can be verified with computer assistance in the following manner. First, verify that the girth of an induced subgraph of the unit 5-hypercube with more edges than vertices is at most 10. Suppose $\Gamma$ is a 4-regular induced subgraph of the 5-dimensional grid of girth at least 12. By Lemma 1, there exists a unit 5-hypercube $H$ such that the intersection between $\Gamma$ and $H$, or any unit 5-hypercube adjacent to $H$, contains the same number of edges and vertices. Let $S$ be the set of all induced subgraphs of a unit 5-hypercube with the same number of edges and vertices, and no cycles of length $\leq 10$. These are the initial candidates for the intersection $\Gamma \cap H$. For each element $h$ of $S$, check if it is possible to create an induced subgraph of a $3 \times 1 \times 1 \times 1 \times 1$ orthotope of girth $\geq 12$ such that the intersection with the centre 5-hypercube is identical to $h$, and the other two are also elements of $S$. If this is not possible, $h$ is eliminated from $S$. A few iterations of this elimination process leaves $S$ empty, and we conclude that the maximal girth is 10.

For $n = 6$, $k = 4$, the girth of the following induced subgraph is 12. The vertex set consists of
$$(x_1, x_2, x_3, x_4, x_5, x_6)$$
$$(x_1 + 1, x_2, x_3, x_4, x_5, x_6)$$
$$(x_1 + 2, x_2, x_3, x_4, x_5, x_6)$$
$$(x_1, x_2, x_3, x_4 + 1, x_5, x_6)$$
$$(x_1, x_2 + 1, x_3, x_4 + 1, x_5, x_6)$$
$$(x_1, x_2 + 2, x_3, x_4 + 1, x_5, x_6)$$
$$(x_1, x_2, x_3, x_4 + 2, x_5, x_6)$$
$$(x_1, x_2, x_3 + 1, x_4 + 2, x_5, x_6)$$
$$(x_1, x_2, x_3 + 2, x_4 + 2, x_5, x_6)$$
for all $(x_1, x_2, x_3, x_4, x_5, x_6) \in \mathbb{Z}^6$ satisfying
$$x_1 + x_2 + x_3 \equiv 0 \textrm{ (mod 3)}$$
$$x_4 + x_5 + x_6 \equiv 0 \textrm{ (mod 3)}$$
$$x_1 + x_4 \equiv x_2 + x_5 \textrm{ (mod 3)}$$

We round off this section with a parameterized example in which $n=k^2$.

\begin{theorem}
	Let $k \geq 3$, and suppose $\{ p_i \}_{1 \leq i \leq 2k}$ is a set of positive integers such that the differences $p_i - p_j$ with $i \neq j$ are all distinct. (For example, one could take $p_i = 2^{i-1}$ $\forall$ $i$.) Let $a_{k (i - 1) + j} = |p_{2i} - p_{2j - 1}|$ for $1 \leq i, j \leq k$. Then the induced subgraph $\Gamma^{(k^2)}$ of the $k^2$-dimensional grid with vertex set given by $$\{(x_1, x_2, ..., x_{k^2}) \in \mathbb{Z}^{k^2} \textrm{ }|\textrm{ } a_1 x_1 + a_2 x_2 + ... + a_{k^2} x_{k^2} \in \{ p_1, p_2, ..., p_{2k}\}\}$$ is $k$-regular of girth 12.
\end{theorem}

\begin{proof}
	It is immediate from the definition of $a_{k (i - 1) + j}$ that $\Gamma^{(k^2)}$ is $k$-regular.
	
	Suppose the girth of $\Gamma^{(k^2)}$ is $g$. Then there is some sequence $$\{ v_0, v_1, ..., v_g = v_0, v_{g+1} = v_1, ..., v_{2g} = v_0 \}$$ of vertices in $\Gamma^{(k^2)}$ such that $v_i$ and $v_{i+1}$ are adjacent and $v_i \neq v_{i+2}$ for $i \leq g-1$. It must be possible to pair each $i \leq g-1$ with some $j \leq g-1$ such that $\overrightarrow{v_i v_{i+1}} = -\overrightarrow{v_j v_{j+1}}$, in order that $$\sum_{i=0}^{g-1} \overrightarrow{v_i v_{i+1}} = 0$$ It follows that the value of $$a_1 x_1 + a_2 x_2 + ... + a_{k^2} x_{k^2}$$ is the same for the coordinates $(x_1, x_2, ...)$ of $v_i$ and for those of $v_{j+1}$; a relation which we denote by $i \sim j+1$ (and accordingly $i+1 \sim j$). Thus, $| i-j |$ is an odd integer. It can not be 1 or 3 due to the restriction $v_i \neq v_{i+2}$, so we have $| i-j | \geq 5$, which implies $g \geq 10$. Moreover, if $\overrightarrow{v_i v_{i+1}} = -\overrightarrow{v_j v_{j+1}}$, then again since $i \sim j+1$,  $\overrightarrow{v_{i+1} v_{i+2}} \neq - \overrightarrow{v_{j+1} v_{j+2}}$ which rules out the possibility $g=10$ (which would require $|i - j| = 5$ for each antipodal pair). On the other hand, for any sequence $\{ v_0, v_1, ..., v_{12} \}$ of adjacent vertices as above in which $$v_0 \sim v_6 \sim v_{12}$$ $$v_1 \sim v_5 \sim v_9$$ $$v_2 \sim v_8$$ $$v_3 \sim v_7 \sim v_{11}$$ $$v_4 \sim v_{10}$$ we necessarily have $v_{12} = v_0$, and having three different classes for the even indexed vertices is possible since $k \geq 3$. We conclude that the girth is equal to 12.
\end{proof}

\section{Non-induced subgraphs}

In this section, the notation $$(x_1^{(1)}, ..., x_n^{(1)}) \sim (x_1^{(2)}, ..., x_n^{(2)})$$ means that there is an edge between $(x_1^{(1)}, ..., x_n^{(1)})$ and $(x_1^{(2)}, ..., x_n^{(2)})$.

It is generally harder to prove anything about subgraphs if they are not required to be induced, but it does open up the possibility of higher values of the girth. We will focus on cubic subgraphs in 3 and 4 dimensions. Each example given here has a vertex set consisting of all lattice points, and then we describe the edge set in each case.

Even when $\Lambda = \mathbb{Z}^3$, the maximal value of the girth is not known. It is either 12, 14 or 16. The upper bound comes from a straightforward comparison of the number of vertices at a graph distance $\leq r$ from a fixed vertex in a 3-regular tree and in the 3-dimensional grid (confer next section), with parity taken into consideration.

For the lower bound, there are infinitely many different cubic subgraphs of girth 12, but we omit the details here. The highest value that has been found for the spread at depth 6 is 84, as attained by the following subgraph $\Gamma^{(3)}$. The edges are given by $$z \equiv 0 \textrm{ (mod 2)} \Rightarrow (x, y, z) \sim (x + 1, y, z)$$ $$z \equiv 1 \textrm{ (mod 2)} \Rightarrow (x, y, z) \sim (x, y + 1, z)$$ $$2x + 2y + 3z \equiv 0, 1 \textrm{ or } 5 \textrm{ (mod 6)} \Rightarrow (x, y, z) \sim (x, y, z + 1)$$

Consider the cubic subgraph $\Gamma ^\textrm{BCC}$ of the grid associated with the BCC lattice, having the following edge set.
\newline

\noindent$x+y \equiv 0 \textrm{ (mod 4)}, z \equiv 0 \textrm{ or } 3 \textrm{ (mod 8)}$

\noindent$\Rightarrow (x, y, z) \sim (x + 1, y + 1, z + 1) \sim (x + 2, y + 2, z + 2) \sim (x + 3, y + 3, z + 3)$
\newline

\noindent$x-y \equiv 0 \textrm{ (mod 4)}, z \equiv 2 \textrm{ or } 5 \textrm{ (mod 8)}$

\noindent$\Rightarrow (x, y, z) \sim (x + 1, y - 1, z + 1) \sim (x + 2, y - 2, z + 2) \sim (x + 3, y - 3, z + 3)$
\newline

\noindent$x+y \equiv 0 \textrm{ (mod 4)}, z \equiv 4 \textrm{ or } 7 \textrm{ (mod 8)}$

\noindent$\Rightarrow (x, y, z) \sim (x - 1, y - 1, z + 1) \sim (x - 2, y - 2, z + 2) \sim (x - 3, y - 3, z + 3)$
\newline

\noindent$x-y \equiv 0 \textrm{ (mod 4)}, z \equiv 6 \textrm{ or } 1 \textrm{ (mod 8)}$

\noindent$\Rightarrow (x, y, z) \sim (x - 1, y + 1, z + 1) \sim (x - 2, y + 2, z + 2) \sim (x - 3, y + 3, z + 3)$
\newline

\noindent We have $g(\Gamma^\textrm{BCC}) = 12$ and $s(\Gamma^\textrm{BCC}, 6) = 93$. Can this be improved upon?

For the grid associated with the FCC lattice, there are infinitely many different cubic subgraphs of girth 14. The following example, $\Gamma ^\textrm{FCC}$, attains a spread of 171 at depth 7. The edge set consists of the edges induced by the vertices $$(4, 0, 0), (3, 1, 0), (0, 2, 0), (0, 1, 1), (3, 2, 1), (4, 3, 1), (1, 1, 2),$$ $$(2, 2, 2), (1, 0, 3), (2, 3, 3), (3, 4, 3), (2, 0, 4), (1, 3, 4)$$ and translated by the vectors spanned by \{[2, 0, 0], [0, 2, 0], [1, 1, 4]\}.

For $\Lambda = \mathbb{Z}^4$, suppose $G$ is a subgraph with the property that for each \newline $i \in \mathbb{Z}$, the adjacency between two vertices $(x_1, x_2, x_3, x_4)$ and $(x_5, x_6, x_7, x_8)$ with $x_1 + x_2 + x_3 + x_4 = i$ and $x_5 + x_6 + x_7 + x_8 = i + 1$ is uniquely determined by $i$ and the residue classes of $x_1, ..., x_8 \textrm{ (mod 2)}$. $G$ can thus be represented by a sequence $\{ H_i \}_{i \in \mathbb{Z}}$ of subgraphs of the unit distance graph on $\{ 0, 1 \}^4$. Suppose furthermore that the set of edges of each $H_i$ is identical to one of the sets $E_1, ..., E_{10}$, in which the neighbours of the vertices $(x_1, x_2, x_3, x_4)$ with $x_1 + x_2 + x_3 + x_4 \equiv 0 \textrm{ (mod 2)}$ via edges in $E_1$, $E_2$, $E_3$ and $E_4$ are given in the following tables, while $E_5 = E_1 \cup E_2$, $E_6 = E_1 \cup E_3$, $E_7 = E_1 \cup E_4$, $E_8 = E_2 \cup E_3$, $E_9 = E_2 \cup E_4$ and $E_{10} = E_3 \cup E_4$:
\newline

\noindent
	\begin{tabular}{|c|c|c|}
		\hline
		$(x_1, x_2, x_3, x_4)$& Adjacent vertex via $E_1$ & Adjacent vertex via $E_2$ \\
		\hline
		$(0, 0, 0, 0) \textrm{ or } (1, 1, 1, 1)$ &
		$(1 - x_1, x_2, x_3, x_4)$ & $(x_1, 1 - x_2, x_3, x_4)$ \\
		$(0, 0, 1, 1) \textrm{ or } (1, 1, 0, 0)$ &
		$(x_1, x_2, x_3, 1 - x_4)$ & $(x_1, x_2, 1 - x_3, x_4)$ \\
		$(0, 1, 0, 1) \textrm{ or } (1, 0, 1, 0)$ &
		$(x_1, 1 - x_2, x_3, x_4)$ & $(1 - x_1, x_2, x_3, x_4)$ \\
		$(0, 1, 1, 0) \textrm{ or } (1, 0, 0, 1)$ &
		$(x_1, x_2, 1 - x_3, x_4)$ & $(x_1, x_2, x_3, 1 - x_4)$ \\
		\hline
	\end{tabular}
\newline
\newline

\noindent
	\begin{tabular}{|c|c|c|}
		\hline
		$(x_1, x_2, x_3, x_4)$& Adjacent vertex via $E_3$ & Adjacent vertex via $E_4$ \\
		\hline
		$(0, 0, 0, 0) \textrm{ or } (1, 1, 1, 1)$ &
		$(x_1, x_2, 1 - x_3, x_4)$ & $(x_1, x_2, x_3, 1 - x_4)$ \\
		$(0, 0, 1, 1) \textrm{ or } (1, 1, 0, 0)$ &
		$(x_1, 1 - x_2, x_3, x_4)$ & $(1 - x_1, x_2, x_3, x_4)$ \\
		$(0, 1, 0, 1) \textrm{ or } (1, 0, 1, 0)$ &
		$(x_1, x_2, x_3, 1 - x_4)$ & $(x_1, x_2, 1 - x_3, x_4)$ \\
		$(0, 1, 1, 0) \textrm{ or } (1, 0, 0, 1)$&
		$(1 - x_1, x_2, x_3, x_4)$ & $(x_1, 1 - x_2, x_3, x_4)$ \\
		\hline
	\end{tabular}
\newline
\newline

\noindent If, for each $i \equiv 0 \textrm{ (mod 2)}$, the edge sets of $(H_i, H_{i+1}, H_{i+2})$ are either
\newline

$(E_5, E_1, E_6), (E_6, E_1, E_7), (E_7, E_1, E_5), (E_5, E_2, E_9)$,

$(E_9, E_2, E_8), (E_8, E_2, E_5), (E_6, E_3, E_8), (E_8, E_3, E_{10})$,

$(E_{10}, E_3, E_6), (E_7, E_4, E_{10}), (E_{10}, E_4, E_9) \textrm{ or } (E_9, E_4, E_7)$
\newline

\noindent then $G$ is cubic with $g(G)=16$ and $s(G, 8) = 368$. There are also other ways to combine these edge sets and obtain a cubic graph of girth 16, but with a lower spread at depth 8.

\section{An upper bound}

Suppose $k$ is a fixed integer $\geq 3$. The maximal girth of a $k$-regular subgraph of the $n$-dimensional grid grows at most linearly with $n$. More specifically:

\begin{theorem}
	If $k \geq 3$ and $n \geq 2$ are integers, $\alpha_k$ is a positive real number such that $$(k - 1)^{\alpha_k} > e(2 \alpha_k + 1)$$ and $r > \alpha_k n$, then the number of vertices at a graph distance $\leq r$ from a fixed vertex is larger in a $k$-regular tree than in the $n$-dimensional grid.
\end{theorem}

\begin{proof}
	The number of vertices at a graph distance $\leq r$ from a fixed vertex in a $k$-regular tree is given by $${k(k - 1)^r - 2} \over {k - 2}$$ The number of vertices at a graph distance $\leq r$ from a fixed vertex in the $n$-dimensional grid is given by the Delannoy number $D(n, r)$. We can assume that the initial vertex is the origin, and then each lattice point that can be reached in $r$ steps is centred in a unit $n$-hypercube that is entirely inside of every half-space given by $$a_1 x_1 + ... + a_n x_n \leq r + {n \over 2}$$ with $a_i = \pm 1$. The hypervolume of the generalized octahedron that forms the intersection of these half-spaces is therefore an upper bound for $D(r, n)$. Thus, we have $$D(r, n) \leq {(2r + n)^n \over n!}$$ Furthermore, we have $$n! \geq \left( {n \over e} \right)^n$$ for $n \geq 0$. This follows directly from the Taylor expansion of $e^x$, as noted in \cite{Tao:2010}. Hence if $(k - 1)^{\alpha_k} > e(2 \alpha_k + 1)$ and $r > \alpha_k n$, we have: $$(k - 1)^r > {{\left(2r + n\right)^n} \over {\left( {n \over e} \right)^n}}$$ $$\Rightarrow {{k(k - 1)^r - 2} \over {k - 2}} > (k - 1)^r > {\left(2r + n\right)^n \over n!} \geq D(r, n)$$
\end{proof}

\begin{corollary}
	With $k$, $n$ and $\alpha_k$ as in Theorem 4, the girth of a $k$-regular subgraph of the $n$-dimensional grid is at most $2 \alpha_k n$.
\end{corollary}

\noindent Some valid $\alpha$-values are: $\alpha_3 = 4.87$, $\alpha_4 = 2.56$, $\alpha_5 = 1.84$.

\section{Conclusion}

We have gone through a number of examples of subgraphs of apparently high girth. Some have been proven to be optimal, but we also have a plethora of open problems.

In addition to considering single cases, we have seen that the smallest possible value of $n$ grows at most quadratically as a function of $k$ if the girth is required to be 12. Is it possible to prove polynomial growth for higher values of the girth (induced or non-induced case)?

On a lighter note, the reader is also encouraged to build physical models of some of these graphs - for aesthetic value.

\end{document}